\documentclass{article}
\usepackage{latexsym, amssymb, amsthm, enumerate, amsmath}
\usepackage{graphicx}
\usepackage[left=1in,top=1in,right=1in,bottom=1in,letterpaper]{geometry}
\usepackage{setspace,url,epsfig}
\usepackage[numbers,sort&compress]{natbib}
\usepackage{algorithm,algorithmic}

\newtheorem{theorem}{Theorem}[section]
\newtheorem*{theorem*}{Theorem} 

\newtheorem*{lemma*}{Lemma}

\begin{document}

\title{Fast Linearized Bregman Iteration for Compressive  Sensing and
Sparse Denoising}
\author{Stanley Osher\thanks{Department of Mathematics, UCLA, Los Angeles, CA 90095
(\texttt{sjo@math.ucla.edu}) This author¡¯s research was supported
by ONR Grant N000140710810, a grant from the Department of Defense
and NIH Grant UH54RR021813} \and Yu Mao\thanks{Department of
Mathematics, UCLA, Los Angeles, CA 90095
(\texttt{ymao29@math.ucla.edu}) This author¡¯s research was
supported by NIH Grant UH54RR021813} \and Bin
Dong\thanks{Department of Mathematics, UCLA, Los Angeles, CA 90095
(\texttt{bdong@math.ucla.edu}) This author¡¯s research was
supported by NIH Grant UH54RR021813} \and Wotao
Yin\thanks{Department of Computational and Applied Mathematics,
Rice University, Houston, TX 77005 (\texttt{wotao.yin@rice.edu})
This author¡¯s research was supported by NSF Grant DMS-0748839 and
an internal faculty research grant from the Dean of Engineering at
Rice University} }

\pagestyle{myheadings} \markboth{Fast Linearized Bregman Iteration
for Compressive Sensing and Sparse Denoising}{Stanley Osher, Yu
Mao, Bin Dong, Wotao Yin}
\date{December 11, 2008}
 \maketitle

\begin{abstract}
We propose and analyze an extremely fast, efficient and simple
method for solving the problem:
\begin{displaymath}
\min\{\|u\|_1 : Au=f, u \in  R^n\}.
\end{displaymath}
This method was first described in \cite{DO}, with more details in
\cite{YOGD} and rigorous theory given in \cite{COS1} and
\cite{COS2}. The motivation was compressive sensing, which now has
a vast and exciting history, which seems to have started with
Candes, et.al. \cite{CRT} and Donoho, \cite{Do2}. See \cite{YOGD},
\cite{COS1} and \cite{COS2} for a large set of references. Our
method introduces an improvement called ``kicking" of the very
efficient method of \cite{DO}, \cite{YOGD} and also applies it  to
the problem of denoising of undersampled signals. The use of
Bregman iteration for denoising of images began in \cite{OBGXY}
and led to improved results for total variation based methods.
Here we apply it to denoise signals, especially essentially sparse
signals, which might even be undersampled.
\end{abstract}

\section{Introduction}

Let $A \in R^{m\times n}$, with $n > m$ and $f \in R^m$, be given.
The aim of a basis pursuit problem is to find $u \in  R^n$ by
solving the constrained minimization problem:
\begin{equation}
\min_{u \in  R^{n}} \{J(u)|Au=f\} \label{1.1}
\end{equation}
where $J(u)$ is a continuous convex function.

For basis pursuit, we take:
\begin{equation}
J(u) = |u|_1 = \sum_{j=1}^n |u_j|. \label{1.2}
\end{equation}
We assume that $AA^T$ is invertible. Thus $Au = f$ is
underdetermined and has at least one solution,
$u=A^T(AA^T)^{-1}f$, which minimizes the $\ell_2$ norm. We also
assume that $J(u)$ is coercive, i.e., whenever $\|u\| \rightarrow
\infty, \ J(u) \rightarrow \infty$. This implies that the set of
all solutions of (\ref{1.1}) is nonempty and convex. Finally, when
$J(u)$ is strictly or strongly convex, the solution of (\ref{1.1})
is unique.

Basis pursuit arises from many applications. In particular, there
has been a recent burst of research in compressive sensing, which
involves solving (\ref{1.1}), (\ref{1.2}). This was led by Candes
et.al. \cite{CRT}, Donoho, \cite{Do2}, and others, see
\cite{YOGD}, \cite{COS1} and \cite{COS2} for extensive references.
Compressive sensing guarantees, under appropriate circumstances,
that the solution to (\ref{1.1}), (\ref{1.2}) gives the sparsest
solution satisfying $Au=f$. The problem then becomes one of
solving (\ref{1.1}), (\ref{1.2}) fast. Conventional linear
programming solvers are not tailored for the large scale dense
matrices $A$ and the sparse solutions $u$ that arise here. To
overcome this, a linearized Bregman iterative procedure was
proposed in \cite{DO} and analyzed in \cite{YOGD}, \cite{COS1} and
\cite{COS2}. In \cite{YOGD}, true, nonlinear Bregman iteration was
also used quite successfully for this problem.

Bregman iteration applied to (\ref{1.1}), (\ref{1.2}) involves
solving the constrained optimization problem through solving a
small number of unconstrained optimization problems:
\begin{equation}
\min_u \left\{\mu |u|_1 + \frac{1}{2} \|Au-f\|_2^2\right\}
\label{1.3}
\end{equation}
for $\mu > 0$.

In \cite{YOGD} we used a method called the fast fixed point
continuation solver (FPC) \cite{HYZ}   which appears to be
efficient. Other solvers of (\ref{1.3}) could be used in this
Bregman iterative regularization procedure.

Here we will improve and analyze a linearized Bregman iterative
regularization procedure, which, in its original incarnation,
\cite{DO}, \cite{YOGD}, involved only a two line code and simple
operations and was already extremely fast and accurate.

In addition, we are interested in the denoising properties of
Bregman iterative regularization, for signals, not images. The
results for images involved the BV norm, which we may discretize
for $n\times n$ pixel images as:
\begin{equation}
TV(u) = \sum_{i,j=1}^{n-1} ((u_{i+1,j}-u_{ij})^2 + (u_{i,j+1} -
u_{ij})^2)^{\frac{1}{2}}. \label{1.4}
\end{equation}
We usually regard the success of the ROF TV based model \cite{ROF}

\begin{equation}
\min_u \left\{TV(u) + \frac{\lambda}{2} \|f-u\|^2\right\}
\label{1.5}
\end{equation}
(we now drop the subscript 2 for the $L_2$ norm throughout the
paper) as due to the fact that images have edges and in fact are
almost
 piecewise constant (with texture added). Therefore, it is not
surprising that sparse signals could be denoised using
(\ref{1.3}). The ROF denoising model was greatly improved in
\cite{OBGXY} and  \cite{CHF}  with the help of Bregman iterative
regularization. We will do the same thing here using Bregman
iteration with (\ref{1.3}) to denoise sparse signals, with the
added touch of undersampling the noisy signals.

The paper is organized as follows: In section 2 we describe
Bregman iterative algorithms, as well as the linearized version.
We motivate these methods and describe previously obtained
theoretical results. In section 3 we introduce an improvement to
the linearized version, call ``kicking" which greatly speeds up
the method, especially for solutions $u$ with a large dynamic
range. In section 4 we present numerical results, including sparse
recovery for $u$ having large dynamic range, and the recovery of
signals in large amounts of noise. In another work in progress
\cite{LOT}   we apply these ideas to denoising very blurry and
noisy signals remarkably well including sparse recovery for $u$.
By blurry we mean situations where $A$ is perhaps a subsampled
discrete convolution matrix whose elements decay to zero with $n$,
e.g. random rows of a discrete Gaussian.

\section{Bregman and Linearized Bregman Iterative Algorithms}

The Bregman distance \cite{Br}, based on the convex function $J$,
between points $u$ and $v$, is defined by
\begin{equation}
D_J^p(u,v) = J(u) - J(v) - \langle p,u-v\rangle \label{2.1}
\end{equation}
where $p \in  \partial J(v)$ is an element in the subgradient of
$J$ at the point $v$. In general $D_J^p(u,v)\not= D_J^p(v,u)$ and
the triangle inequality is not satisfied, so $D_J^p(u,v)$ is not a
distance in the usual sense. However it does measure the closeness
between $u$ and $v$ in the sense that $D_J^p(u,v) \geq 0$ and
$D_J^p(u,v) \geq D_J^p(w,v)$ for all points $w$ on the line
segment connecting $u$ and $v$. Moreover, if $J$ is convex,
$D_J^p(u,v) \geq 0$, if $J$ is strictly convex $D_J^p(u,v) > 0$
for $u \not= v$ and if $J$ is strongly convex, then there exists a
constant $a > 0$ such that
\[
D_J^p(u,v) \geq a\|u-v\|^2.
\]

To solve (\ref{1.1}) Bregman iteration was proposed in \cite{YOGD}
 . Given $u^0 = p^0 = 0$, we define:
\begin{align}
u^{k+1} &= \arg\min_{u\in R^{n}} \left\{J(u)-J(u^k) - \langle
u-u^k,p^k\rangle + \frac{1}{2}
\|Au-f\|^2\right\}  \label{2.2}\\
p^{k+1} &= p^k - A^T (Au^{k+1}-f). \nonumber
\end{align}
This can be written as
\begin{displaymath}
u^{k+1} = \arg\min_{u\in R^{2}} \left\{D_J^{p^{k}} (u,u^k) +
\frac{1}{2} \|Au-f\|^2\right\}. \nonumber
\end{displaymath}

It was proven in \cite{YOGD}   that if $J(u) \in  C^2 (\Omega)$
and is strictly convex in $\Omega$, then $\|Au^k-f\|$ decays
exponentially whenever $u^k \in  \Omega$ for all $k$. Furthermore,
when $u^k$ converges, its limit is a solution of (\ref{1.1}). It
was also proven in \cite{YOGD}   that when $J(u) = |u|_1$, i.e.
for problem (\ref{1.1}) and (\ref{1.2}), or when $J$ is a convex
function satisfying some additional conditions, the iteration
(\ref{2.2}) leads to a solution of (\ref{1.1}) in finitely many
steps.

As shown in \cite{YOGD}, see also \cite{OBGXY}, \cite{CHF}, the
Bregman iteration (\ref{2.2}) can be written as:
\begin{align}
f^{k+1} &= f^k + f-Au^k \nonumber \\
u^{k+1} &= \arg\min_{u\in R^{n}} \left\{J(u) + \frac{1}{2}
\|Au-f^{k+1}\|^2\right\} \label{2.3}
\end{align}
This was referred to as ``adding back the residual" in
\cite{OBGXY}  . Here $f^0 = 0, u^0 = 0$. Thus the Bregman
iteration uses solutions of the unconstrained problem
\begin{equation}
\min_{u\in R} \left\{J(u) + \frac{1}{2} \|Au-f\|^2\right\}
\label{2.4}
\end{equation}
as a solver in which the Bregman iteration applies this process
iteratively.

Since there is generally no explicit expression for the solver of
(\ref{2.2}) or (\ref{2.3}), we turn to iterative methods. The
linearized Bregman iteration which we will analyze, improve and
use here is generated by
\begin{align}
u^{k+1} &= \arg\min_{u\in R^{n}} \left\{J(u)-J(u^k) - \langle
u-u^k,p^k\rangle +
\frac{1}{2\delta} \|u-(u^k - \delta A^T(Au^k-f))\|^2\right\} \nonumber \\
p^{k+1} &= p^k - \frac{1}{\delta} (u^{k+1} - u^k) - A^T(Au^k-f)
\label{2.5}.
\end{align}

In the special case considered here, where $J(u) = \mu \|u\|_1$,
then we have the two line algorithm
\begin{align}
v^{k+1} &= v^k - A^T (Au^{k} - f) \label{2.6.1}\\
u^{k+1} &= \delta \cdot \text{shrink}  (v^{k+1},\mu) \label{2.6.2}
\end{align}
where $v^k$ is an auxiliary variable
\begin{equation}
v^k = p^{k} + \frac{1}{\delta} u^{k} \label{2.6}
\end{equation}
and
\begin{align*}
\text{shrink}  (x,\mu): =
\begin{cases} x-\mu, & \text{if} \ x > \mu \\
0, & \text{if} \ -\mu \leq x \leq \mu \\
x + \mu, & \text{if} \ x < -\mu
\end{cases}
\end{align*}
is the soft thresholding algorithm \cite{Do1}  .

This linearized Bregman iterative algorithm was invented in
\cite{DO}   and used and analyzed in \cite{YOGD},\cite{COS1} and
\cite{COS2}. In fact it comes from the inner-outer iteration for
(\ref{2.2}). In \cite{YOGD}   it was shown that the linearized
Bregman iteration (\ref{2.5}) is just one step of the inner
iteration for each outer iteration. Here we repeat the arguments
also in \cite{YOGD}, which begin by summing the second equation in
(\ref{2.5}) arriving at (using the fact that $u^0=p^0=0$):
\begin{eqnarray}
&p^k + \dfrac{1}{\delta} u^k + \sum_{j=0}^{k-1} A^T(Au^j-f)
\nonumber  = p^k + \dfrac{1}{\delta}  u^k - v^{k} = 0, \
\text{for} \ k=1,2,\ldots.
\end{eqnarray}
This gives us (\ref{2.6.2}), and allows us to rewrite its first
equation as:
\begin{equation}
u^{k+1} = \arg\min_{u\in R^{n}} \left\{J(u) + \frac{1}{2\delta}
\|u - \delta v^{k+1}\|^2\right\} \label{2.7}
\end{equation}
i.e. we are adding back the ``linearized noise", where $v^{k+1}$
is defined in (\ref{2.6.1}).

In \cite{YOGD} and \cite{COS1} some interesting analysis was done
for (\ref{2.5}), (and some for (\ref{2.7})). This was done first
for $J(u)$ continuously differentiable in (\ref{2.5}) and the
gradient $\partial J(u)$ satisfying
\begin{equation}
\|\partial J(u) - \partial J(v)\|^2 \leq \beta \langle \partial
J(u) - \partial J(v), u-v\rangle, \label{2.8}
\end{equation}
$\forall u,v \in  R^n, \ \beta > 0$.
In \cite{COS1} it was shown that, if (\ref{2.8}) is true, then
both of the sequences $(u^k)_{k\in N}$ and $(p^k)_{k\in N}$
defined by (\ref{2.5}) converge for $0 < \delta <
\frac{2}{\|AA^T\|}$.

In \cite{COS2} the authors recently give a theoretical analysis,
showing that the iteration in (\ref{2.6.1}) and (\ref{2.6.2})
converges to the unique solution of

\begin{equation}\min_{u\in R^{n}} \left\{\mu\|u\|_1 + \frac{1}{2\delta}
\|u\|^2: Au=f\right\} \label{2.11}
\end{equation}
They also show the interesting result: let $S$ be the set of all
solutions of the Basis Pursuit problem (\ref{1.1}), (\ref{1.2})
and let
\begin{equation}
u_1 = \arg\min_{u\in S} \|u\|^2 \label{2.16}
\end{equation}
which is unique. Denote the solution of (\ref{2.11}) as $u_\mu^*$.
Then
\begin{equation}
\lim_{\mu\rightarrow\infty} \|u_{\mu}^* - u_1\| = 0. \label{2.17}
\end{equation}
In passing they show that
\begin{equation}
\|u_{   \mu}^*\| \leq \|u_1\| \ \text{for all} \ \mu > 0
\label{2.18}
\end{equation}
which we will use below.

Another theoretical analysis on Linearized Bregman algorithm is
given by Yin in \cite{Yinprivate}, where he shows that Linearized
Bregman iteration is equivalent to gradient descent applied to the
dual of the problem (\ref{2.11}) and uses this fact to obtain an
elegant convergence proof.

This summarizes the relevant convergence analysis for our Bregman
and linearized Bregman models.

Next we recall some results from \cite{OBGXY}   regarding noise
and Bregman iteration.

For any sequence $\{u^k\}, \{p^k\}$ satisfying (\ref{2.2}) for $J$
continuous and convex, we have, for any $\tilde\mu$
\begin{eqnarray}
D_J^{p^{k}} (\tilde u,u^k) - D_Jp^{k-1} (\tilde u, u^{k-1}) \leq
\langle A\tilde u - f, Au^{k-1}-f\rangle - \|Au^{k-1}-f\|^2.
\label{2.19}
\end{eqnarray}

Besides implying that the Bregman distance between $u^k$ and any
element $\tilde u$ satisfying $A\tilde u = f$ is monotonically
decreasing, it also implies that, if $\tilde u$ is the ``noise
free" approximation to the solution of (\ref{1.1}), the Bregman
distance between $u^k$ and $\tilde u$ diminishes as long as
\begin{equation}
\|Au^k - f \| > \|A\tilde u - f\| = \sigma, \ \text{where} \
\sigma \ \text{is some measure of the noise} \label{2.20}
\end{equation}
i.e., until we get too close to the noisy signal in the sense of
(\ref{2.20}). Note, in \cite{OBGXY}   we took $A$ to be the
identity, but these more general results are also proven there.
This gives us a stopping criterion for our denoising algorithm.

In \cite{OBGXY}   we obtained a result for linearized Bregman
iteration, following \cite{Ba}, which states that the Bregman
distance between $\tilde u$ and $u^k$ diminish as long as
\begin{equation}
\|A\tilde u - f\| < (1-2\delta \|A A^T\|) \ \|Au^k-f\|
\label{2.21}
\end{equation}
so we need $0 < 2\delta \|A A^T\| < 1$.

In practice, we will use (\ref{2.20}) as our stopping criterion.

\section{Convergence}

We begin with the following simple results for the linearized
Bregman iteration or the equivalent algorithm (\ref{2.5}).

\begin{theorem}\label{thm1}
If $u^k \rightarrow u^\infty$, then $A u^\infty = f$.
\end{theorem}
\begin{proof}
 Assume $Au^\infty \not= f$.  Then $A^T(Au^\infty - f)
\not= 0$ since $A^T$ has full rank. This means that for some $i$,
$(A^T(Au^k - f))_i$ converges to a nonzero value, which means that
$v_i^{k+1}-v_i^k$ does as well. On the other hand $\{v^k\} = \{
u^k/\delta + p^k\}$ is bounded since $\{u^k\}$ converges and $p^k
\in [-\mu,\mu]$. Therefore $\{v_i^k\}$ is bounded,  while
$v_i^{k+1}-v_i^k$ converges to a nonzero limit, which is
impossible.
 \end{proof}
\begin{theorem}\label{thm2}
If $u^k \rightarrow u^\infty$ and $v^k \rightarrow v^\infty$, then
$u^\infty$ minimizes $\{J(u) + \frac{1}{2\delta} \|u\|^2: Au =
f\}$.
\end{theorem}

\begin{proof} Let $\tilde J(u) = J(u) + \frac{1}{2\delta} \|u\|^2$.
then
\begin{displaymath}
\partial \tilde J(u) = \partial J(u) + \frac{1}{\delta} u.
\end{displaymath}
Since $\partial J(u^k) = p^k = v^k - u^k/\delta$, we have
$\partial \tilde J(u^k) = v^k$. Using the non-negativity of the
Bregman distance we obtain
\begin{align*}
\tilde J(u^k) &\leq \tilde J(u_{\text{opt}}) -
\langle u_{\text{opt}} - u^k, \partial \tilde J (u^k)\rangle \nonumber \\
&= \tilde J(u_{\text{opt}}) - \langle u_{\text{opt}} - u^k,
v^k\rangle \nonumber
\end{align*}
where $u_{\text{opt}}$ minimizes (\ref{1.1}) with $J$ replaced by
$\tilde J$, which is strictly convex.

Let $k \rightarrow \infty$, we have
\begin{displaymath}
\tilde J(u^\infty) \leq \tilde J(u_{\text{opt}}) - \langle
u_{\text{opt}} - u^\infty, v^\infty \rangle
\end{displaymath}
Since $v^k = A^T \sum_{j=0}^{k-1} A^T (f - Au^j)$, we have
$v^\infty \in \text{range}(A^T)$. Since $A u_{\text{opt}} =
Au^\infty = f$, we have $\langle u_{\text{opt}} - u^\infty,
v^\infty\rangle = 0$, which implies $\tilde J(u^\infty) \leq
\tilde J(u_{\text{opt}})$.
 \end{proof}
Equation (\ref{2.11}) (from a result in \cite{COS1}  ) implies
that $u^\infty$ will approach a solution to (\ref{1.1}),
(\ref{1.2}), as $\mu$ approaches $\infty$.

The linearized Bregman iteration has the following monotonicity
property:

\begin{theorem}\label{thm3}
If $u^{k+1} \not= u^k$ and $0 < \delta < 2/\|AA^T\|$, then
\begin{displaymath}
\|Au^{k+1} - f\| < \| Au^k-f\|.
\end{displaymath}
\end{theorem}

\begin{proof}  Let
\begin{displaymath}
u^{k+1} - u^k = \Delta u^k, \ v^{k+1} - v^k = \Delta v^k.
\end{displaymath}
Then the shrinkage operation is such that
\begin{equation}
\Delta u_i^k = \delta q_i^k \Delta v_i^{k} \label{3.1}
\end{equation}
with
\begin{align*}
q_i^k \begin{cases}=1&\text{if} \ u_i^{k+1} u_i^k > 0  \\
=0 &\text{if} \ u_i^{k+1} = u_i^k = 0 \\
\in (0,1] &\text{otherwise}
\end{cases}
\end{align*}
Let $Q^k = \text{Diag}~ (q_i^k)$. Then (\ref{3.1}) can be written
as
\begin{equation}
\Delta u^k = \delta Q^k \Delta v^{k} = \delta Q^k A^T(f-Au^k)
\label{3.2}
\end{equation}
which implies
\begin{equation}
Au^{k+1}-f = (I-\delta A Q^kA^T)(Au^k-f). \label{3.3}
\end{equation}

From (\ref{3.1}), $Q^k$ is diagonal with $0 \preceq Q^k \preceq
I$, so $0 \preceq AQ^kA^T \preceq AA^T$. If we choose $\delta>0$
such that $\delta AA^T \prec 2I$, then $0 \preceq \delta A Q^kA^T
\prec 2I$ or $-I \prec I-\delta AQ^kA^T \preceq I$ which implies
that $\|Au^k-f\|$ is not increasing. To get strict decay, we need
only show that $AQ^kA^T(Au^k-f) = 0$ is impossible if
 $u^{k+1} \not= u^k$. Suppose $AQ^kA^T(Au^k-f) = 0$ holds, then from (\ref{3.2}) we have:
\begin{displaymath}
\langle \Delta u^k, \Delta v^{k}\rangle = \delta \langle
A^T(f-Au^k), Q^k A^T(f-Au^k)\rangle=0.
\end{displaymath}
By (\ref{3.1}), this only happens if $u_i^{k+1} = u_i^k$ for all
$i$, which is a contradiction. \end{proof}

We are still faced with estimating how fast the residual decays.
It turns out that if consecutive elements of $u$ do not change
sign, then $\|Au-f\|$ decays exponentially. By 'exponential' we
mean that the ratio of the residuals of two consecutive iteration
converges to a constant, this type of convergence is sometimes
called linear convergence.  Here we define
\begin{equation}
S_u = \{x \in  R^n : \text{sign} (x_i) = \text{sign} (u_i),
\forall i\} \label{3.4}
\end{equation}
(where $\text{sign}(0)=0$ and $\text{sign}(a)=a/|a|$ for $a\neq
0$). Then we have the following:

\begin{theorem}\label{thm4}
If $u^k \in  S \equiv S_{u_{k}}$ for $k \in( T_1,T_2)$, then $u^k$
converges to $u^*$, where $u^* \in \arg\min \{\|Au-f\|^2:u \in
S\}$ and  $\|Au^k-f\|^2$ decays to $\|Au^{*}-f\|^2$ exponentially.
\end{theorem}

\begin{proof}. Since $u^k \in  S$ for $k \in  [T_1,T_2]$, we can
define $Q \equiv Q^k$ for $T_1 \leq k \leq T_2 - 1$. From
(\ref{3.1}) we see that $Q^k$ is a diagonal matrix consisting of
zeros or ones, so $Q = Q^TQ$. Moreover, it is easy to see that $S
= \{x | Qx=x\}$.

Following the argument in Theorem \ref{thm3} we have:
\begin{eqnarray}
&u^{k+1} - u^k =\Delta u^k = \delta Q \Delta v^{k} =
\delta Q A^T (f-Au^k) \label{3.5} \\
&Au^{k+1}-f =[I - \delta AQA^T] (Au^k-f) \label{3.6}
\end{eqnarray} and\[
 -I \prec I-\delta AQA^T \preceq I.
\]

Let $R^n = V_0 \oplus V_1$, where $V_0$ is the null space of
$AQA^T$ and $V_1$ is spanned by the eigenvectors corresponding to
the nonzero eigenvalues of $AQA^T$. Let $Au^k-f = w^{k,0} +
w^{k,1}$, where $w^{k,j} \in V_j$ for $j=0,1$. From (\ref{3.6}) we
have
\begin{eqnarray}
w^{k+1,0} &=& w^{k,0} \nonumber \\
w^{k+1,1} &=& [I-\delta AQA^T] w^{k,1} \nonumber
\end{eqnarray}
for $T_1 \leq k \leq T_2 - 1$. Since $w^{k,1}$ is not in the null
space of $AQA^T$, then (\ref{3.5}) and (\ref{3.6}) imply that
$\|w^{k,1}\|$ decays exponentially. Let $w^0 = w^{k,0}$, then
$AQA^T w^0 = 0$ $AQQA^Tw^0 \Rightarrow QA^T w^0 = 0$. Therefore,
from (\ref{3.5}) we have
\begin{displaymath}
\Delta u^k = \delta Q^TA^T(f-Au^k) = \delta QA^T(w^0+w^{k,1}) =
\delta QA^Tw^{k,1}.
\end{displaymath}
Thus $\|\Delta u^k\|$ decays exponentially. This means $\{u^k\}$
forms a Cauchy sequence in $S$, so it has  a limit $u^* \in S$.
Moreover
\begin{displaymath}
Au^*-f = \lim_k (Au^k-f) = \lim_k w^{k,0} + \lim_k w^{k,1} = w^0.
\end{displaymath}
Since $V_0$ and $V_1$ are orthogonal:
\begin{displaymath}
\|Au^k-f\|^2 = \|w^{k,0}\|^2 + \|w^{k,1}\|^2 = \|Au^* - f\|^2 +
\|w^{k,1}\|^2,
\end{displaymath}
so $\|Au^k-f\|^2 - \|Au^*-f\|^2$ decays exponentially. The only
thing left to show is that
\begin{displaymath}
u^* = \arg\min (\|Au-f\|^2: u\in  S) = \arg\min \{\|Au-f\|^2:
Qu=u\}.
\end{displaymath}
This is equivalent to way that $A^T(Au^*-f)$ is orthogonal with
the hyperspace $\{u:Qu=u\}$. It's easy to see that since $Q$ is a
projection operator, a vector $v$ is orthogonal with $\{u:Qu=u\}$
if and only if $Qv=0$, thus we need to show $QA^T(Au^*-f)=0$. This
is obvious because we have shown that $Au^* - f = w^0$ and
$QA^Tw^0 = 0$. So we find that $u^*$ is the desired minimizer.
\end{proof}

Therefore, instead of decaying exponentially with a global rate,
the residual of the linearized Bregman iteration decays in a
rather sophisticated manner. From the definition of the shrinkage
function 
we can see that the sign of an element of $u$ will change if and
only if the corresponding element of $v$ crosses the boundary of
the interval $[-\mu,\mu]$. If $\mu$ is relatively large compared
with the size of $\Delta v$ (which is usually the case when
applying the algorithm to a compressed sensing problem), then at
most iterations the signs of the elements of $u$ will stay
unchanged, i.e. $u$ will stay in the subspace $S_u$ defined in
(\ref{3.4}) for a long while. This theorem tells us that under
this scenario $u$ will quickly converge to the point $u^*$ that
minimizes $\|Au-f\|$ inside $S_u$, and the difference between
$\|Au-f\|$ and $\|Au^*-f\|$ decays exponentially. After $u$
converges to $u^*$, $u$ will stay there until the sign of some
element of $u$ changes. Usually this means that a new nonzero
element of $u$ comes up.  After that, $u$ will enter a different
subspace $S$ and a new converging procedure begins.

The phenomenon described above can be observed clearly in Fig
\ref{Fig:withoutkicking}. The final solution of $u$ contains five
non-zero spikes. Each time a new spike appears, it converges
rapidly to the position that minimizes $\|Au-f\|$ in the subspace
$S_u$. After that there is a long stagnation, which means $u$ is
just waiting there until the accumulating $v$ brings out a new
non-zero element of $u$. The larger $\mu$ is, the longer the
stagnation takes. Although the convergence of the residual during
each phase is fast, the total speed of the convergence suffers
much from the stagnation. The solution of this problem will be
described in the next section.

\begin{figure}[ht]
    \centering
    \includegraphics[width=3.2in]{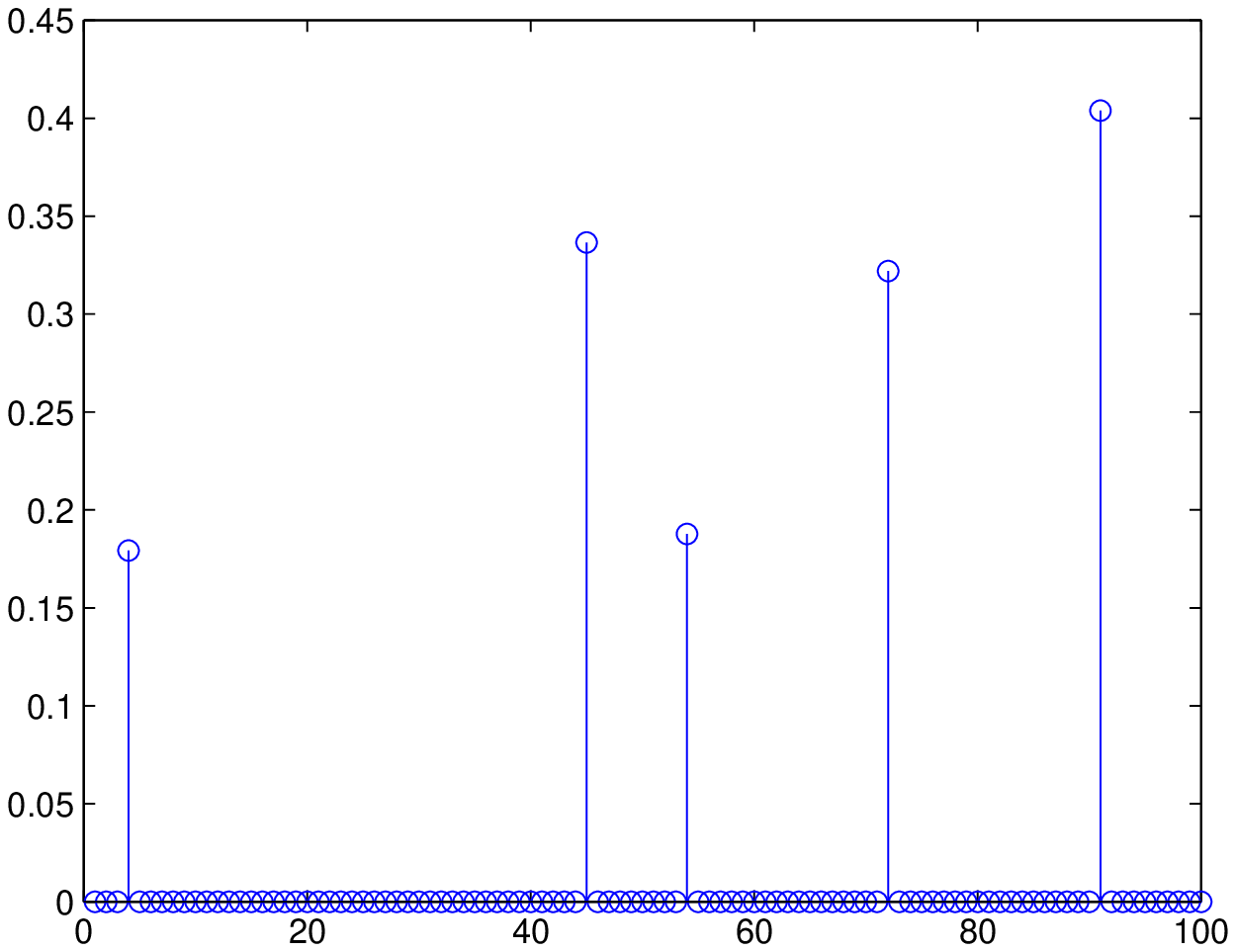}\includegraphics[width=3.0in]{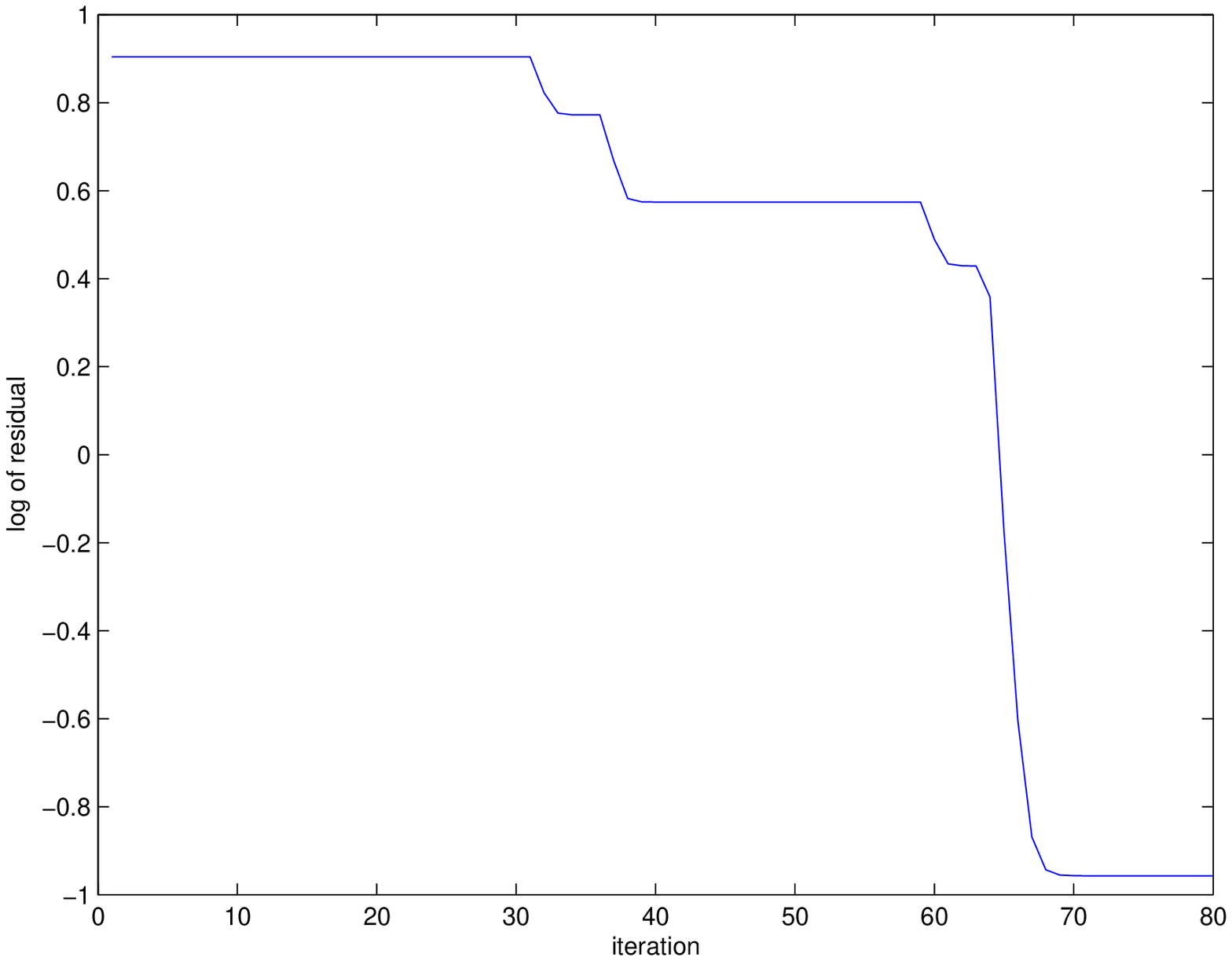}
    \caption{The left figure presents a simple signal with 5 non-zero spikes. The right figure shows how the linearized
    Bregman iteration converges.
    }\label{Fig:withoutkicking}
\end{figure}

\section{Fast Implementation}

The iterative formula in Algorithm \ref{Breg_Alg} below gives us
the basic linearized Bregman algorithm designed to solve
(\ref{1.1}),(\ref{1.2}).

\begin{algorithm}
\caption{Bregman Iterative Regularization\label{Breg_Alg}}
\begin{algorithmic}[Linearized Bregman Iteration]
\STATE Initialize: $u=0$, $v=0$.

\WHILE{``$\|f-Au\|$ not converge"}

\STATE $v^{k+1}=v^k + A^\top (f-Au^k)$

\STATE $u^{k+1}=\delta\cdot \text{shrink}(v^{k+1},\mu)$

\ENDWHILE
\end{algorithmic}
\end{algorithm}

This is an extremely concise algorithm,  simple to program,
involve only matrix multiplication and shrinkage. When $A$
consists of rows of a matrix of a fast transform like FFT which is
a common case for compressed sensing, it is even faster because
matrix multiplication
 can be implemented efficiently using the existing fast
code of the transform. Also, storage becomes a less serious issue.

We now consider how we can accelerate the algorithm under the
problem of stagnation described  in the previous section. From
that discussion, during a stagnation $u$ converges to a limit
$u^*$
 so we will have $u^{k+1}\approx u^{k+2}\approx\cdots\approx u^{k+m}\approx
u^*$ for some $m$. Therefore the increment of $v$ in each step,
$A^\top(f-Au)$, is fixed. This implies that during the stagnation
$u$ and $v$ can be calculated explicitly as following
\begin{equation}
\begin{cases}\label{stagnation}
u^{k+j}\equiv u^{k+1}\\
 v^{k+j}=v^k+j\cdot A^\top(f-Au^{k+1})
\end{cases}j=1,\cdots,m
\end{equation}
If we denote the set of indices of the zero elements of $u^*$ as
$I_0$ and let $I_1=\overline{I_0}$ be the support of $u^*$, then
$v^k_i$ will keep changing only for $i\in I_0$ and the iteration
can be formulated entry-wise as:
\begin{equation}
\begin{cases}\label{entrywisestagnation}
u^{k+j}_i\equiv u^{k+1}_i&\forall i\\
 v^{k+j}_i=v^k_i+j\cdot(A^\top(f-Au^{k+1}))_i &i\in I_0
\\
 v^{k+j}_i\equiv v^{k+1}_i&i\in I_1
\end{cases}
\end{equation}
for $j=1,\cdots,m$. The stagnation will end when $u$ begins to
change again. This happens if and only if some
 element of $v$ in $I_0$ (which keeps changing during the stagnation)
crosses the boundary of the
 interval $[-\mu,\mu]$. When $i\in I_0$, $v^k_i\in[-\mu,\mu]$, so
  we can estimate  the number of the steps needed for
 $v_i^k$ to cross the boundary $\forall i\in I_0$ from
\eqref{entrywisestagnation}, which
 is
\begin{equation}\label{entrywisestep}
 s_i=\left\lceil\dfrac{\mu\cdot\text{sign}((A^\top(f-Au^{k+1}))_i)-v^{k+1}_i}{(A^\top(f-Au^{k+1}))_i}\right\rceil\,\forall i\in I_0
\end{equation}and \begin{equation}\label{kickingstep}
s=\min_{i\in I_0}\{s_i\}\end{equation} is the number of steps
needed. Therefore, $s$ is nothing but the length of the
stagnation. Using \eqref{stagnation}, we can predict the end
status of the stagnation by\begin{equation}
\begin{cases}\label{kicking}
u^{k+s}\equiv u^{k+1}\\
 v^{k+s}=v^k+s\cdot A^\top(f-Au^{k+1})
\end{cases}j=1,\cdots,m
\end{equation}
Therefore, we can \emph{kick} $u$ to the critical point of the
stagnation when we detect that $u$ has been staying unchanged for
a while. Specifically, we have the following algorithm: Algorithm
 \ref{Breg_Alg_Kick}.
\begin{algorithm}
\caption{Linearized Bregman Iteration with
Kicking\label{Breg_Alg_Kick}}
\begin{algorithmic}[Linearized Bregman Iteration with Kicking]
\STATE Initialize: $u=0$, $v=0$.

\WHILE{``$\|f-Au\|$ not converge"}

\IF{``$u^{k-1}\approx u^k$"}

\STATE calculate $s$ from \eqref{entrywisestep} and
\eqref{kickingstep}

\STATE $v^{k+1}_i=v^k_i+s\cdot (A^\top(f-Au^k))_i$, $\forall i \in
I_0$

\STATE $v^{k+1}_i=v^k_i$, $\forall i \in I_1$

\ELSE

\STATE $v^{k+1}=v^k+ A^\top (f-Au^{k})$

\ENDIF

\STATE $u^{k+1}=\delta\cdot\text{shrink}(v^{k+1},\mu)$

\ENDWHILE
\end{algorithmic}
\end{algorithm}

\begin{figure}[ht]
    \centering
    \includegraphics[width=3.0in]{withoutkicking.eps}\includegraphics[width=3.0in]{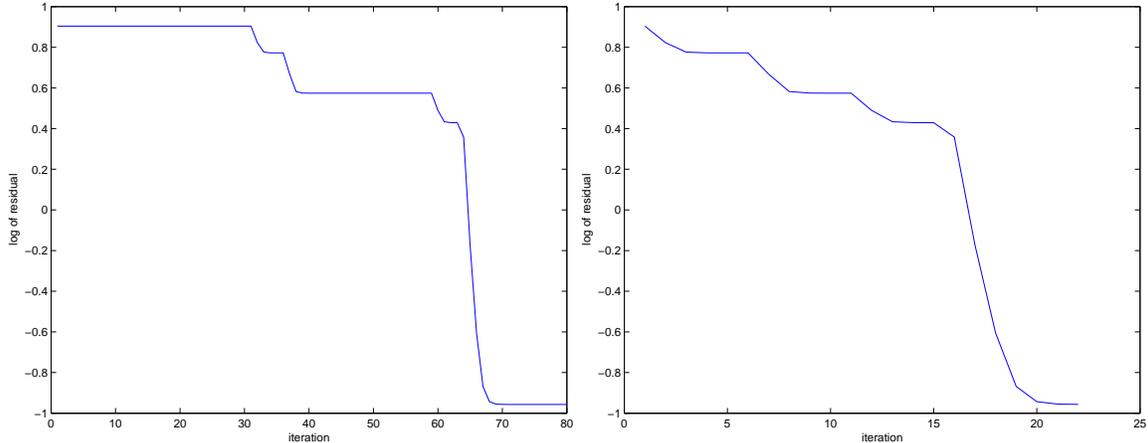}
    \caption{The left figure presents the convergence curve of
the original linearized
    Bregman iteration  using the same signal as Fig
    \ref{Fig:withoutkicking}. The right figure shows the convergence
curve of the  linearized
    Bregman iteration with the kicking modification.
    }\label{Fig:kicking}
\end{figure}
Indeed, this kicking procedure is similar to line search commonly
used in optimization problems and modifies the initial algorithm
in no way but just accelerates the speed. More precisely, note
that the output sequence $\{u^k, v^k\}$ is a subsequence of the
original one, so all the previous theoretical conclusions on
convergence still hold here.

An example of the algorithm is shown in Fig \ref{Fig:kicking}. It
is clear that all the stagnation in the original convergence
collapses to single steps. The total amount of computation is
reduced dramatically.

\section{Numerical Results}
In this section, we demonstrate the effectiveness of the algorithm
(with kicking) in solving basis pursuit and some related problems.

\subsection{Efficiency}

Consider the constrained minimization problem
$$\min|u|_{1}\quad
\mbox{s.t. } Au=f,$$ where the constraints $Au=f$ are
under-determined linear equations with $A$ an $m\times n$ matrix,
and $f$ generated from a sparse signal $\bar{u}$ that has a number
of nonzeros $\kappa<m$.

Our numerical experiments use two types of $A$ matrices: Gaussian
matrices whose elements were generated from i.i.d. normal
distributions $\mathcal{N}(0,1)$ (\textbf{randn(m,n)} in MATLAB),
and partial discrete cosine transform (DCT) matrices whose $k$
rows were chosen randomly from the $n\times n$ DCT matrix. These
matrices are known to be efficient for compressed sensing. The
number of rows $m$ is chosen as $m\sim \kappa\log(n/\kappa)$ for
Gaussian matrices and $m\sim \kappa\log n$ for DCT matrices
(following \cite{CRT}  ).

The tested \emph{original sparse signals} $\bar{u}$ had numbers of
nonzeros equal to $0.05n$ and $0.02n$ rounded to the nearest
integers in two sets of experiments, which were obtained by
\textbf{round(0.05*n)} and \textbf{round(0.02*n)} in MATLAB,
respectively. Given a sparsity $\|\bar{u}\|_0$, i.e., the number
of nonzeros, an \emph{original sparse signal}
$\bar{u}\in\mathbb{R}^n$ was generated by randomly selecting the
locations of $\|\bar{u}\|_0$ nonzeros, and sampling each of these
nonzero elements from $\mathcal{U}(-1,1)$ (\textbf{2*(rand-0.5)}
in MATLAB). Then, $f$ was computed as $A\bar{u}$. When
$\|\bar{u}\|_0$ is small enough, we expect the basis pursuit
problem, which we solved using our fast algorithm, to yield a
solution $u^*=\bar{u}$ from the inputs $A$ and $f$.

Note that partial DCT matrices are implicitly stored fast
transforms for which matrix-vector multiplications in the forms of
$Ax$ and $A^\top x$ were computed by the MATLAB commands
\textbf{dct(x)} and \textbf{idct(x)}, respectively. Therefore, we
were able to test on partial DCT matrices of much larger sizes
than Gaussian matrices. The sizes $m$-by-$n$ of these matrices are
given in the first two columns of Table \ref{tb:tol_e-5}.

Our code was written in MATLAB and was run on a Windows PC with a
Intel(R) Core(TM) 2 Duo 2.0GHz CPU and 2GB memory. The MATLAB
version is 7.4.

The set of computational results given in Table \ref{tb:tol_e-5}
was obtained by using the stopping criterion
\begin{equation}\label{tol_e-5} \frac{\|Au^k-f\|}{\|f\|}<10^{-5},
\end{equation} which was sufficient to give a small error
$\|u^k-\bar{u}\|/\|\bar{u}\|$. Throughout our experiments in Table
\ref{tb:tol_e-5}, we used $\mu=1$ to ensure the correctness of the
results.


\begin{table}[ht]
\caption{Experiment results using 10 random instances for each
configuration of $(m,n,\|\bar{u}\|_0)$, with nonzero elements of
$\bar u$ come from $\mathcal{U}(-1,1)$.} \label{tb:tol_e-5}
\vspace{10pt}
\begin{center}
\begin{tabular}{r|r||c|c|c||c|c|c||c|c|c}\hline\hline
\multicolumn{11}{c}{Results of linearized Bregman-$L_1$ with
kicking}\\\hline \multicolumn{3}{l|}{Stopping tolerance.} &
\multicolumn{8}{l}{$\|Au^k-f\|/\|f\|<10^{-5}$}\\\hline \hline
\multicolumn{2}{c||}{ } & \multicolumn{9}{c}{Gaussian
matrices}\\\hline
 & & \multicolumn{3}{c||}{stopping itr. $k$} & \multicolumn{3}{c||}{relative error $\|u^k-\bar{u}\|/\|\bar{u}\|$} & \multicolumn{3}{c}{time (sec.)} \\ \cline{3-11}
 & & mean & std. & max & mean & std. & max & mean & std. & max \\\cline{3-11}
$n$ & $m$ & \multicolumn{9}{c}{$\|\bar{u}\|_0 = 0.05n$}\\\hline
 1000 & 300 & 422 & 67 & 546 & 2.0e-05 & 4.3e-06 & 2.7e-05 & 0.42 & 0.06 & 0.51 \\
 2000 & 600 & 525 & 57 & 612 & 1.8e-05 & 1.9e-06 & 2.1e-05 & 4.02 & 0.45 & 4.72 \\
 4000 & 1200 & 847 & 91 & 1058 & 1.7e-05 & 1.7e-06 & 1.9e-05 & 25.7 & 2.87 & 32.1 \\ \hline
$n$ & $m$ & \multicolumn{9}{c}{$\|\bar{u}\|_0 = 0.02n$}\\\hline
 1000 & 156 & 452 & 98 & 607 & 2.3e-05 & 2.6e-06 & 2.6e-05 & 0.24 & 0.06& 0.33 \\
 2000 & 312 & 377 & 91 & 602 & 2.0e-05 & 4.0e-06 & 2.9e-05 & 1.45 & 0.38 & 2.37 \\
 4000 & 468 & 426 & 30 & 477 & 1.6e-05 & 2.1e-06 & 2.0e-05 & 6.96 & 0.51 & 7.94 \\ \hline
\hline \multicolumn{2}{c||}{ } &\multicolumn{9}{c}{Partial DCT
matrices}\\\hline $n$ & $m$ &  \multicolumn{9}{c}{$\|\bar{u}\|_0 =
0.05n$}\\\hline
 4000 & 2000 & 71 & 6.6 & 82 & 9.1e-06 & 2.5e-06 & 1.2e-05 & 0.43 & 0.06 & 0.56\\
 20000 & 10000 & 158 & 14.5 & 186 & 6.2e-06 & 2.1e-06 & 1.1e-05 & 3.95 & 0.36 & 4.73 \\
 50000 & 25000 & 276 & 14 & 296 & 6.8e-06 & 2.6e-06 & 1.0e-05 & 17.6 & 0.99 & 19.2\\ \hline
$n$ & $m$ & \multicolumn{9}{c}{$\|\bar{u}\|_0 = 0.02n$}\\\hline
 4000 & 1327 & 52 & 7.0 & 64 & 8.6e-06 & 1.3e-06 & 1.1e-05 & 0.27 & 0.04 & 0.35 \\
 20000 & 7923 & 91 & 10.3 & 115 & 7.2e-06 & 2.2e-06 & 1.1e-05 & 2.36 & 0.30 & 3.02\\
 50000 & 21640 & 140 & 9.7 & 153 & 5.9e-06 & 2.4e-06 & 1.1e-05 & 8.53 & 0.66 & 9.42\\ \hline
\end{tabular}
\end{center}
\end{table}

\subsection{Robustness to Noise}\label{subs:rubost}

In real applications, the measurement $f$ we obtain is usually
contaminated by noise. The measurement we have is:
$$\tilde f=f+n=A\bar u+n,\quad n\in
\mathcal{N}(0,\sigma).$$ To characterize the noise level, we shall
use SNR (signal to noise ratio) instead of $\sigma$ itself. The
SNR is defined as follows
$$SNR(u):=20\log_{10}(\frac{\|\bar
u\|}{\|n\|}).$$ In this section we test our algorithm on
recovering the true signal $\bar u$ from $A$ and the noisy
measurement $\tilde f$. As in the last section, the nonzero
entries of $\bar u$ are generated from $\mathcal{U}(-1,1)$, and
$A$ is either a Gaussian random matrix or a partial DCT matrix.
Our stopping criteria is given by $$\mbox{std}\big(Au^{k}-\tilde
f\big)<\sigma,\quad\mbox{and}\quad\mbox{Iter.}<1000,$$ i.e. we
stop whenever the standard deviation of residual $Au^k-\tilde f$
is less than $\sigma$ or the number of iterations exceeds $1000$.
Table \ref{tb:noisy} shows numerical results for different noise
level, size of $A$ and sparsity. We also show one typical result
for a partial DCT matrix with size $n=4000$ and $\|\bar
u\|_0=0.02n=80$ in Figure \ref{Fig:noisyexample}.

\begin{table}[ht]
\caption{Experiment results using 10 random instances for each
configuration of $(m,n,\|\bar{u}\|_0)$.} \label{tb:noisy}
\vspace{10pt}
\begin{center}
\begin{tabular}{r|r||c|c|c||c|c|c||c|c|c}\hline\hline
\multicolumn{11}{c}{Results of linearized Bregman-$L_1$ with
kicking}\\\hline \multicolumn{3}{l|}{Stopping criteria.} &
\multicolumn{8}{l}{$\mbox{std}(Au^k-f)<\sigma$.}\\\hline \hline
\multicolumn{2}{c||}{ } & \multicolumn{9}{c}{Gaussian
matrices}\\\hline
 & & \multicolumn{3}{c||}{stopping itr. $k$} & \multicolumn{3}{c||}{relative error $\|u^k-\bar{u}\|/\|\bar{u}\|$} & \multicolumn{3}{c}{time (sec.)} \\ \cline{3-11}
 & & mean & std. & max & mean & std. & max & mean & std. & max \\\cline{3-11}
 Avg. SNR& ($n$,$m$) & \multicolumn{9}{c}{$\|\bar{u}\|_0 = 0.05n$}\\\hline
  26.12& (1000,300) & 420 & 95 & 604 & 0.0608 & 0.0138 & 0.0912 & 0.33 & 0.09 & 0.53 \\
  25.44& (2000,600) & 206 & 32 & 253 & 0.0636 & 0.0128 & 0.0896 & 1.49 & 0.22 & 1.79 \\
  26.02& (4000,1200) & 114 & 11 & 132 & 0.0622 & 0.0079 & 0.0738 & 3.32 & 0.31 & 3.81 \\ \hline
 Avg. SNR& ($n$,$m$) & \multicolumn{9}{c}{$\|\bar{u}\|_0 =
0.02n$}\\\hline
  27.48& (1000,156) & 890 & 369 & 1612 & 0.0456 & 0.0085 & 0.0599 & 0.42 & 0.17& 0.73 \\
  25.06& (2000,312) & 404 & 64 & 510 & 0.0638 & 0.0133 & 0.0843 & 1.37 & 0.23 & 1.74 \\
  26.04& (4000,468) & 216 & 35 & 267 & 0.0557 & 0.0068 & 0.0639 & 3.29 & 0.55 & 4.13 \\ \hline
\hline \multicolumn{2}{c||}{ } &\multicolumn{9}{c}{Partial DCT
matrices}\\\hline Avg. SNR& ($n$,$m$) &
\multicolumn{9}{c}{$\|\bar{u}\|_0 = 0.05n$}\\\hline
  23.97& (4000, 2000) & 151 & 9.2 & 170 & 0.0300 & 0.0028 & 0.0332 & 0.94 & 0.07 & 1.03\\
  24.00& (20000,10000) & 250 & 14 & 270 & 0.0300 & 0.0010 & 0.0318 & 7.88 & 0.62 & 8.86 \\
  24.09& (50000,25000) & 274 & 9.9 & 295 & 0.0304 & 0.0082 & 0.0315 & 20.4 & 0.74 & 20.1\\ \hline
Avg. SNR& ($n$,$m$) & \multicolumn{9}{c}{$\|\bar{u}\|_0 =
0.02n$}\\\hline
  24.29& (4000,1327) & 130 & 11 & 157 & 0.0223 & 0.0023 & 0.0253 & 0.79 & 0.08 & 1.00 \\
  24.37& (20000,7923) & 223 & 14 & 257 & 0.0204 & 0.0025 & 0.0242 & 6.89 & 0.53 & 8.15\\
  24.16& (50000,21640) & 283 & 19 & 311 & 0.0193 & 0.0012 & 0.0207 & 21.5 & 1.68 & 24.1\\ \hline
\end{tabular}
\end{center}
\end{table}
\begin{figure}[ht]
    \centering
    \includegraphics[width=5.0in]{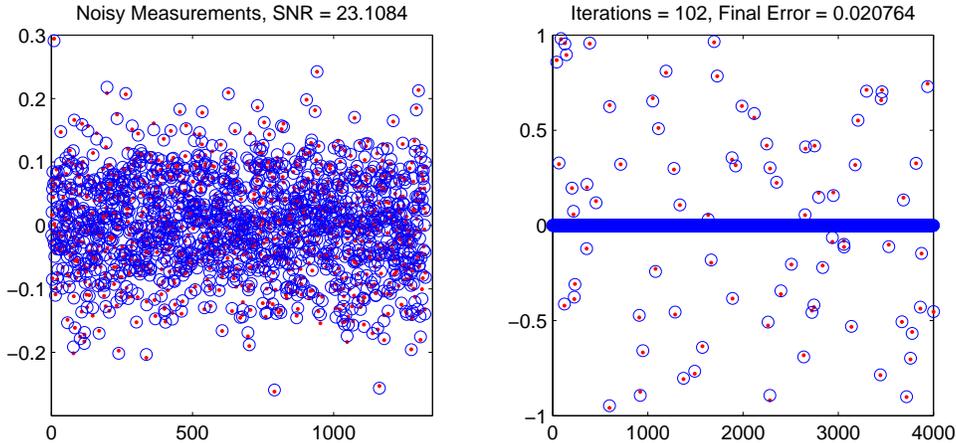}
    \caption{The left figure presents the clean (red dots) and noisy (blue circles) measurements, with
    SNR=23.1084; the right figure shows the reconstructed signal (blue circles) v.s. original signal (red dots),
    where the relative error=0.020764, and number of iterations is 102.}\label{Fig:noisyexample}
\end{figure}

\subsection{Recovery of Signal with High Dynamical Range}
In this section, we test our algorithm on signals with high
dynamical ranges. Precisely speaking, let $\mbox{MAX}=\max\{|\bar
u_i|: i=1,\ldots,n\}$ and $\mbox{MIN}=\min\{|u_i|: u_i\ne0,
i=1,\ldots,n\}$. The signals we shall consider here satisfy
$\frac{\mbox{MAX}}{\mbox{MIN}}\approx 10^{10}$. Our $\bar u$ is
generated by multiplying a random number in $[0,1]$ with another
one randomly picked from $\{1,10,\ldots,10^{10}\}$. Here we adopt
the stopping criteria $$\frac{\|Au^k-f\|}{\|f\|}<10^{-11}$$ for
the case without noise (Figure \ref{Fig:dynamic}) and the same
stopping criteria as in the previous section for the noisy cases
(Figures \ref{Fig:dynamic:noisy1}-\ref{Fig:dynamic:noisy3}). In
the experiments, we take the dimension $n=4000$, the number of
nonzeros of $\bar u$ to be $0.02n$, and $\mu=10^{10}$. Here $\mu$
is chosen to be much larger than before, because the dynamical
range of $\bar u$ is large. Figure \ref{Fig:dynamic} shows results
for the noise free case, where the algorithm converges to a
$10^{-11}$ residual in less than 300 iterations. Figures
\ref{Fig:dynamic:noisy1}-\ref{Fig:dynamic:noisy3} show the cases
with noise (the noise is added the same way as in previous
section). As one can see, if the measurements are contaminated
with less noise, signals with smaller magnitudes will be recovered
well. For example in Figure \ref{Fig:dynamic:noisy1}, the
SNR$\approx 118$, and the entries of magnitudes $10^4$ are well
recovered; in Figure \ref{Fig:dynamic:noisy2}, the SNR$\approx
97$, and the entries of magnitudes $10^5$ are well recovered; and
in Figure \ref{Fig:dynamic:noisy3}, the SNR$\approx 49$, and the
entries of magnitudes $10^7$ are well recovered.

\begin{figure}[ht]
    \centering
    \includegraphics[width=6.0in]{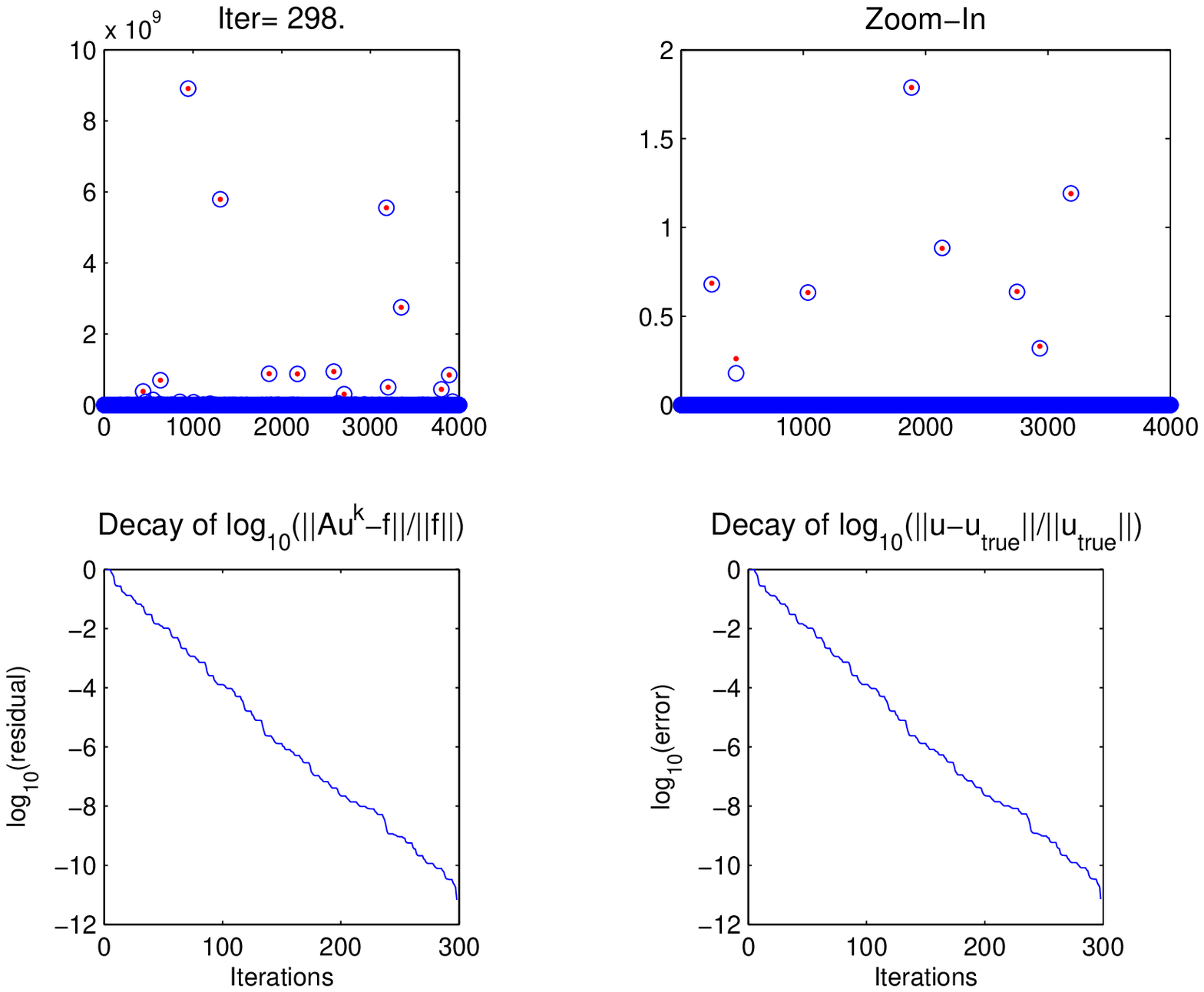}
    \caption{Upper left, true signal (red dots) v.s. recovered signal (blue circle); upper right, one zoom-in
    to the lower magnitudes; lower left, decay of residual $\log_{10}\frac{\|Au^k-f\|}{\|f\|}$; lower right,
    decay of error to true solution $\log_{10}\frac{\|u^k-\bar u\|}{\|\bar
    u\|}$.}\label{Fig:dynamic}
\end{figure}
\begin{figure}
    \centering
    \includegraphics[width=4.5in]{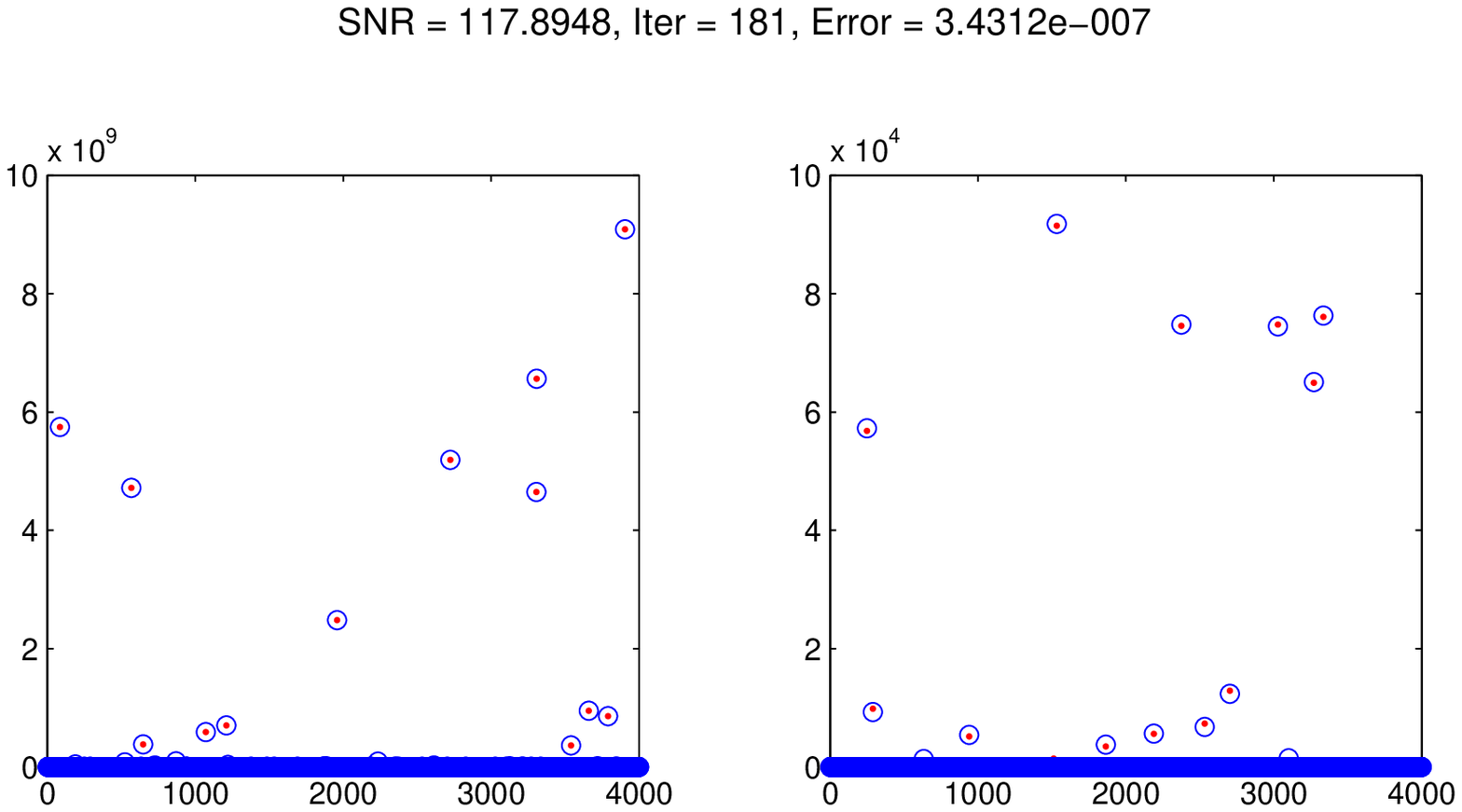}
    \caption{Noisy case. Left figure, true signal (red dots) v.s. recovered signal (blue circle); right figure, one zoom-in
    to the magnitude$\approx 10^5$. The error is measured by $\frac{\|u^k-\bar u\|}{\|\bar
    u\|}$.}\label{Fig:dynamic:noisy1}
\end{figure}
\begin{figure}
    \centering
    \includegraphics[width=4.5in]{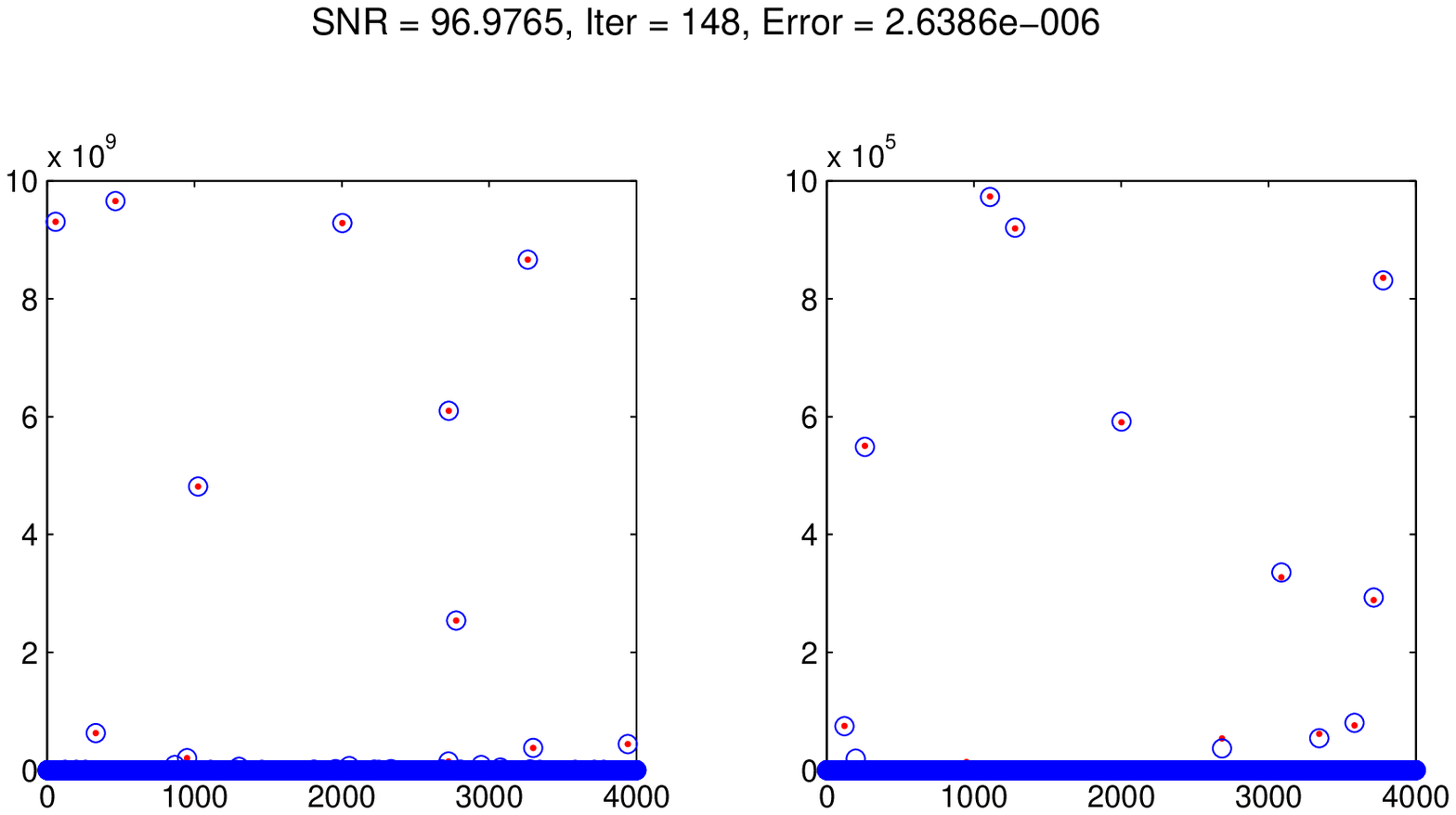}
    \caption{Noisy case. Left figure, true signal (red dots) v.s. recovered signal (blue circle); right figure, one zoom-in
    to the magnitude$\approx 10^6$. The error is measured by $\frac{\|u^k-\bar u\|}{\|\bar
    u\|}$.}\label{Fig:dynamic:noisy2}
\end{figure}
\begin{figure}
    \centering
    \includegraphics[width=4.5in]{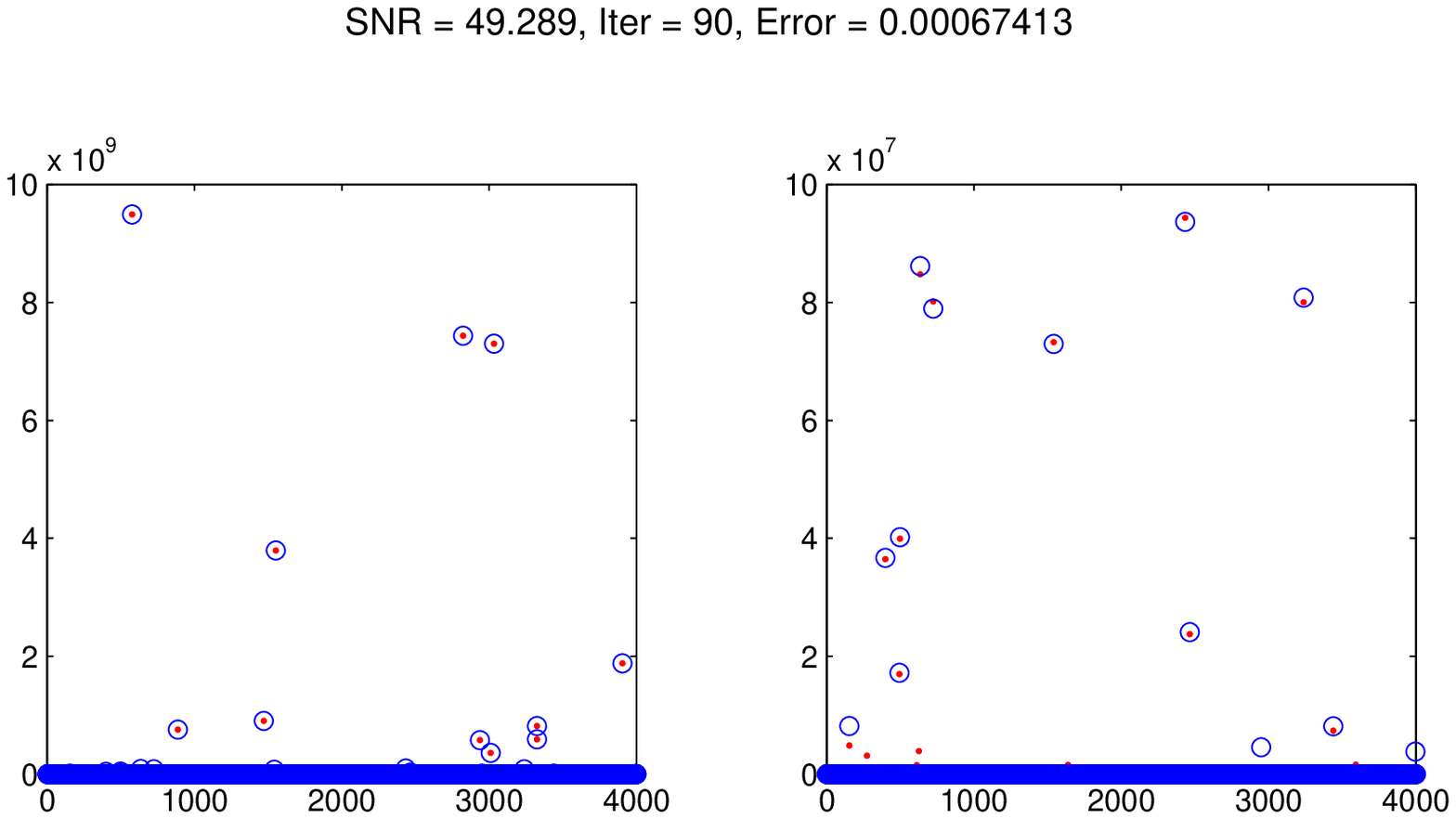}
    \caption{Noisy case. Left figure, true signal (red dots) v.s. recovered signal (blue circle); right figure, one zoom-in
    to the magnitude$\approx 10^8$. The error is measured by $\frac{\|u^k-\bar u\|}{\|\bar
    u\|}$.}\label{Fig:dynamic:noisy3}
\end{figure}

\subsection{Recovery of Sinusoidal Waves in Huge Noise}

In this section we consider $$\bar u(t)=a\sin(\alpha
t)+b\cos(\beta t),$$ where $a,b,\alpha$ and $\beta$ are unknown.
The observed signal $\tilde u$ is noisy and has the form $\tilde
u=\bar u+n$ with $n\sim\mathcal{N}(0,\sigma)$. In practice, the
noise in $\tilde u$ could be huge, i.e. possibly have a negative
SNR, and we may only be able to observe partial information of
$\tilde u$, i.e. only a subset of values of $\tilde u$ is known.
Notice that the signal is sparse (only four spikes) in frequency
domain. Therefore, this is essentially a compressed sensing
problem and $\ell_1$-minimization should work well here. Now the
problem can be stated as reconstructing the original signal $\bar
u$ from \textit{random samples} of the observed signal $\tilde u$
using our fast $\ell_1$-minimization algorithm. In our
experiments, the magnitudes $a$ and $b$ are generated from
$\mathcal{U}(-1,1)$; frequencies $\alpha$ and $\beta$ are random
multiples of $\frac{2\pi}{n}$, i.e. $\alpha=k_1\frac{2\pi}{n}$ and
$\alpha=k_2\frac{2\pi}{n}$, with $k_i$ taken from
$\{0,1,\ldots,n-1\}$ randomly and $n$ denotes the dimension. We
let $I$ be a random subset of $\{1,2,\ldots,n\}$ and $f=\tilde
u(I)$, and take $A$ and $A^\top$ to be the partial matrix of
inverse Fourier matrix and Fourier matrix respectively. Now we
perform our algorithm adopting the same stopping criteria as in
section \ref{subs:rubost}, and obtain a reconstructed signal
denoted as $x$. Notice that reconstructed signal $x$ is in Fourier
the domain, not in the physical domain. Thus we take an inverse
Fourier transform to get the reconstructed signal in physical
domain, denoted as $u^*$. Since we know a priori that our solution
should have four spikes in Fourier domain, before we take the
inverse Fourier transform, we pick the four spikes with largest
magnitudes and set the rest of the entries to be zero. Some
numerical results are given in Figure
\ref{Fig:sinu1}-\ref{Fig:sinu4}. Our experiments show that the
larger the noise level is, the more random samples we need for a
reliable reconstruction, where reliable means that with high
probability ($>$80\%) of getting the frequency back exactly. As
for the magnitudes $a$ and $b$, our algorithm cannot guarantee to
recover them exactly (as one can see in Figure
\ref{Fig:sinu1}-\ref{Fig:sinu4}). However, frequency information
is much more important than magnitudes in the sense that the
reconstructed signal is less sensitive to errors in magnitudes
than errors in frequencies (see bottom figures in Figure
\ref{Fig:sinu1}-\ref{Fig:sinu4}). On the other hand, once we
recover the right frequencies, one can use hardware to estimate
magnitudes accurately.

\section{Conclusion}
We have proposed the linearized Bregman iterative algorithms as a
competitive method for solving the compressed sensing problem.
Besides the simplicity of the algorithm, the special structure of
the iteration enables the kicking scheme to accelerate the
algorithm even when $\mu$ is extremely large. As a result, a
sparse solution can always be approached efficiently.

It also turns out that our process has remarkable denoising
properties for undersampled sparse signals. We will pursue this in
further work.

Our results suggest there is a big category of problem that  can
be solved by linearized Bregman iterative algorithms. We hope that
our method and its extensions could produce even more applications
for problems under different scenarios, including very
underdetermined inverse problems in partial differential
equations.

\section{Acknowledgements}

S.O. was supported by ONR Grant N000140710810, a grant from the
Department of Defense and NIH Grant UH54RR021813; Y.M. and B.D.
were supported by NIH Grant UH54RR021813; W.Y.  was supported by
NSF Grant DMS-0748839 and an internal faculty research grant from
the Dean of Engineering at Rice University.

\bibliography{biblist}

\begin{thebibliography}{10}

\bibitem{DO}
J.~Darbon and S.~Osher.
\newblock Fast discrete optimizations for sparse approximations and
  deconvolutions.
\newblock preprint 2007.

\bibitem{YOGD}
W.~Yin, S.~Osher, D.~Goldfarb, and J.~Darbon.
\newblock Bregman iterative algorithms for compressed sensing and related
  problems.
\newblock {\em SIAM J. Imaging Sciences 1(1).}, pages 143--168, 2008.

\bibitem{COS1}
J.~Cai, S.~Osher, and Z.~Shen.
\newblock Linearized {B}regman iterations for compressed sensing.
\newblock {\em Math. Comp.}, 2008.
\newblock to appear, see also {U}CLA CAM Report 08-06.

\bibitem{COS2}
J.~Cai, S.~Osher, and Z.~Shen.
\newblock Convergence of the linearized {B}regman iteration for $\ell_1$-norm
  minimization.
\newblock {U}CLA CAM Report 08-52, 2008.

\bibitem{CRT}
E.~Candes, J.~Romberg, and T.~Tao.
\newblock {Robust uncertainty principles: exact signal reconstruction from
  highly incomplete frequency information}.
\newblock 52(2):489--509, 2006.

\bibitem{Do2}
D.L. Donoho.
\newblock {Compressed sensing}.
\newblock {\em IEEE Trans. Inform. Theory}, 52:1289--1306, 2006.

\bibitem{OBGXY}
S.~Osher, M.~Burger, D.~Goldfarb, J.~Xu, and W.~Yin.
\newblock {An iterative regularization method for total variation based image
  restoration}.
\newblock {\em Multiscale Model. Simul}, 4(2):460--489, 2005.

\bibitem{HYZ}
E.~Hale, W.~Yin, and Y.~Zhang.
\newblock A fixed-point continuation method for $\ell_1$-regularization with
  application to compressed sensing.
\newblock {C}AAM Technical Report TR07-07, Rice University, Houston, TX, 2007.

\bibitem{ROF}
L.~Rudin, S.~Osher, and E.~Fatemi.
\newblock {Nonlinear total variation based noise removal algorithms}.
\newblock {\em Phys. D}, 60:259--268, 1992.

\bibitem{CHF}
T-C. Chang, L.~He, and T.~Fang.
\newblock Mr image reconstruction from sparse radial samples using bregman
  iteration.
\newblock {\em Proceedings of the 13th Annual Meeting of ISMRM}, 2006.

\bibitem{LOT}
Y.~Li, S.~Osher, and Y.-H. Tsai.
\newblock Recovery of sparse noisy date from solutions to the heat equation.
\newblock in preparation.

\bibitem{Br}
L.M. Bregman.
\newblock {The relaxation method of finding the common point of convex sets and
  its application to the solution of problems in convex programming}.
\newblock {\em USSR Computational Mathematics and Mathematical Physics},
  7(3):200--217, 1967.

\bibitem{Do1}
D.L. Donoho.
\newblock {De-noising by soft-thresholding}.
\newblock {\em IEEE Trans. Inform. Theory}.

\bibitem{Yinprivate}
W.~Yin.
\newblock On the linearized bregman algorithm.
\newblock private communication.

\bibitem{Ba}
M.~Bachmayr.
\newblock {Iterative total variation methods for nonlinear inverse problems}.
\newblock {\em Master's thesis, Johannes Kepler Universit\"at, Linz, Austria},
  2007.

\end{thebibliography}

\begin{figure}
    \centering
    \includegraphics[width=6.0in]{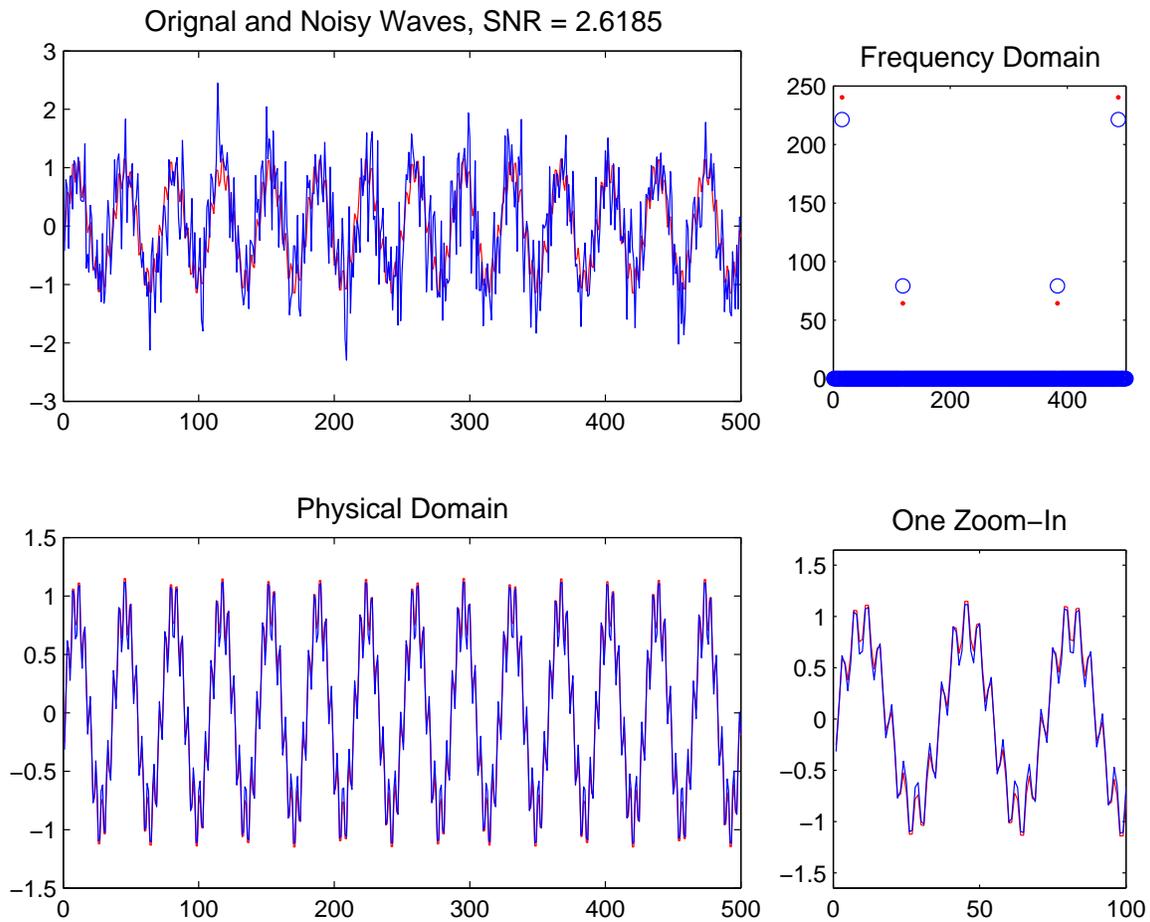}
    \caption{Reconstruction using 20\% random samples of $\tilde u$ with
SNR$=2.6185$. The upper left figure
    shows the original (red) and noisy (blue) signals; the upper right shows the reconstruction (blue circle) v.s. original signal
    (red dots) in Fourier domain in terms of their magnitudes (i.e. $|\widehat u^*|$ v.s. $|\widehat {\bar u}|$);
    bottom left shows the reconstructed (blue) v.s. original (red) signal in physical domain; and bottom right shows
    one close-up of the figure at bottom left.}\label{Fig:sinu1}
\end{figure}

\begin{figure}
    \centering
    \includegraphics[width=6.0in]{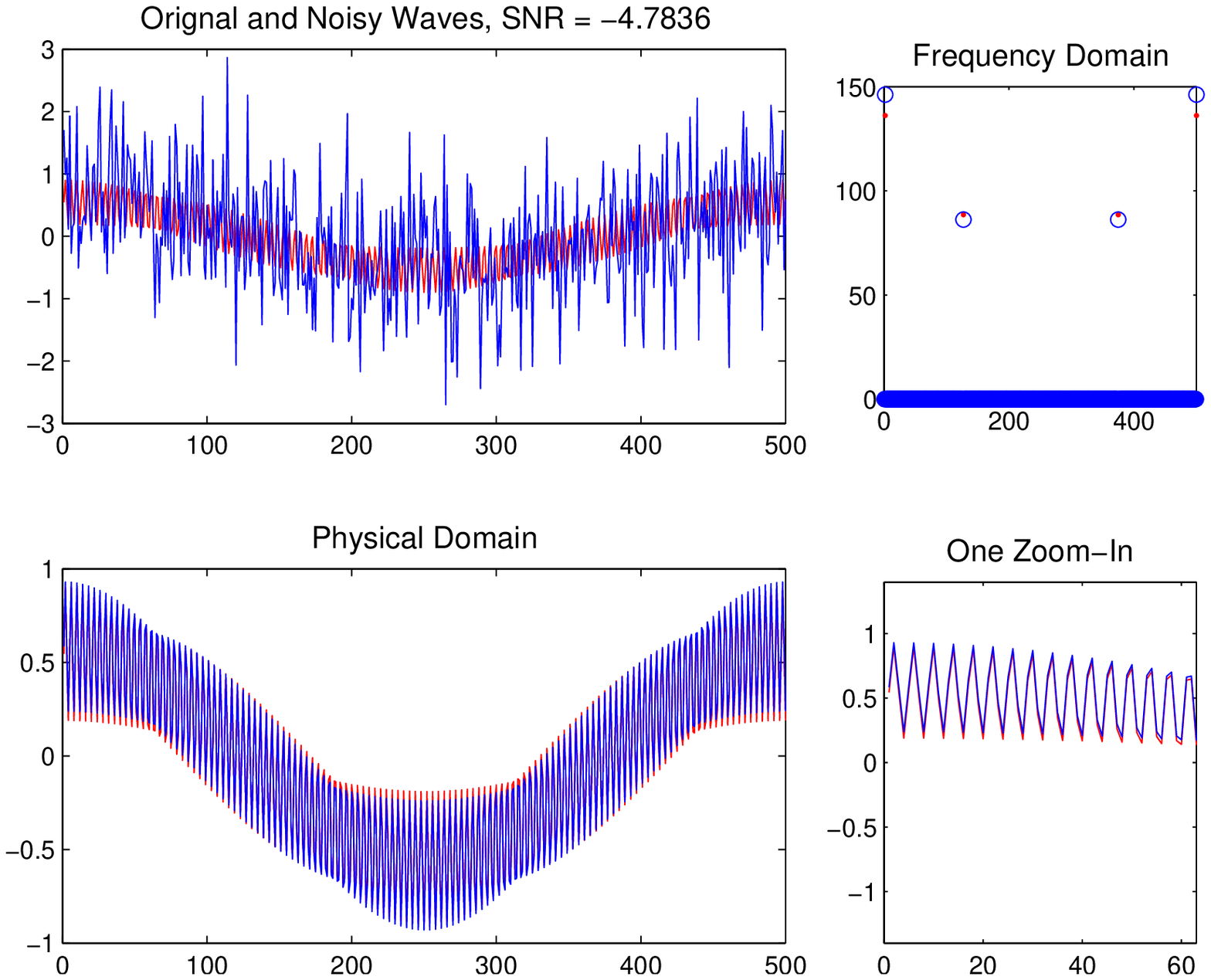}
    \caption{Reconstruction using 40\% random samples of $\tilde u$
with SNR$=-4.7836$. The upper left figure
    shows the original (red) and noisy (blue) signals; the upper right shows the reconstruction (blue circle) v.s. original signal
    (red dots) in Fourier domain in terms of their magnitudes (i.e. $|\widehat u^*|$ v.s. $|\widehat {\bar u}|$);
    bottom left shows the reconstructed (blue) v.s. original (red) signal in physical domain; and bottom right shows
    one close-up of the figure at bottom left.}\label{Fig:sinu2}
\end{figure}

\begin{figure}
    \centering
    \includegraphics[width=6.0in]{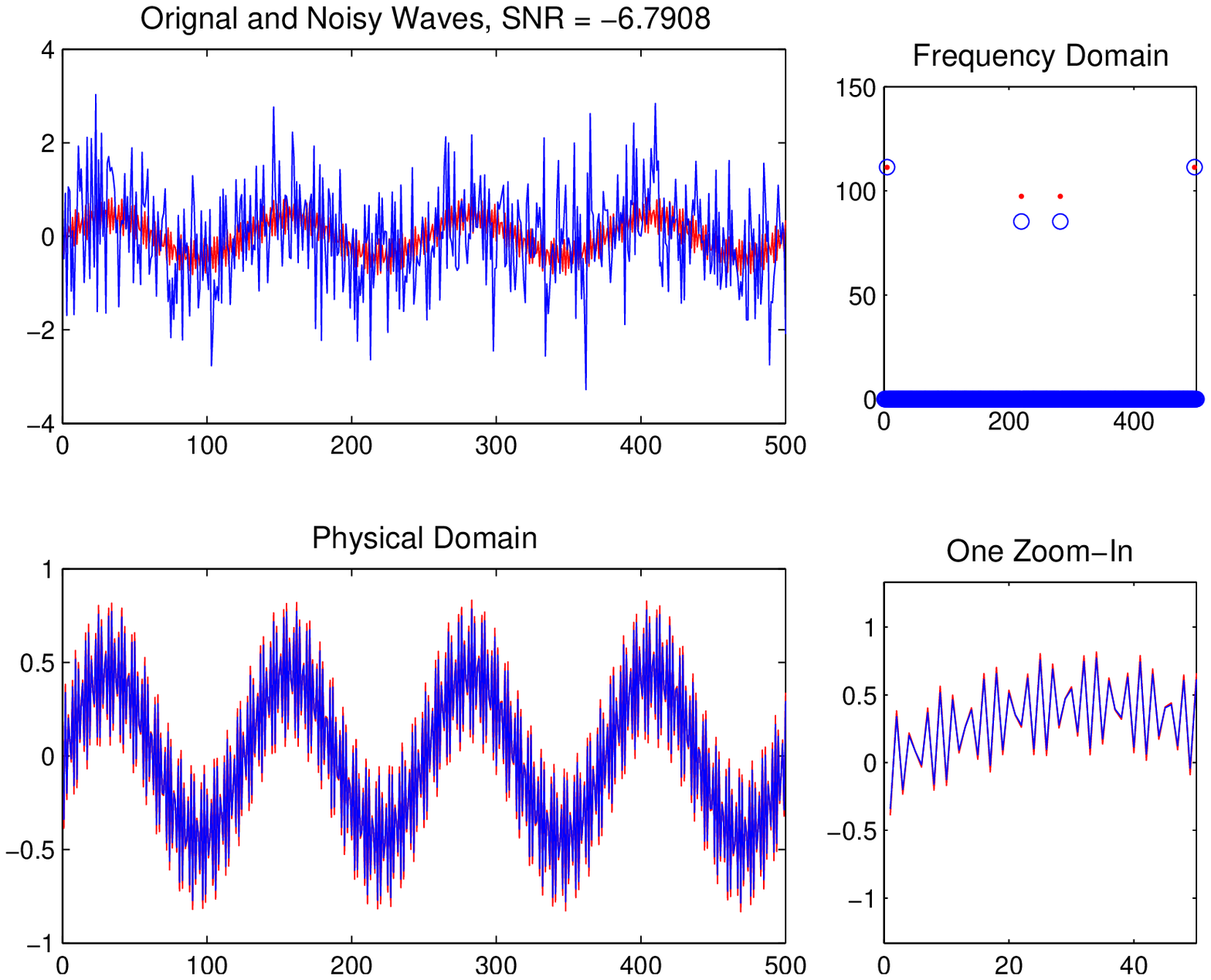}
    \caption{Reconstruction using 60\% random samples of $\tilde u$
with SNR$=-6.7908$. The upper left figure
    shows the original (red) and noisy (blue) signals; the upper right shows the reconstruction (blue circle) v.s. original signal
    (red dots) in Fourier domain in terms of their magnitudes (i.e. $|\widehat u^*|$ v.s. $|\widehat {\bar u}|$);
    bottom left shows the reconstructed (blue) v.s. original (red) signal in physical domain; and bottom right shows
    one close-up of the figure at bottom left.}\label{Fig:sinu3}
\end{figure}
\begin{figure}
    \centering
    \includegraphics[width=6.0in]{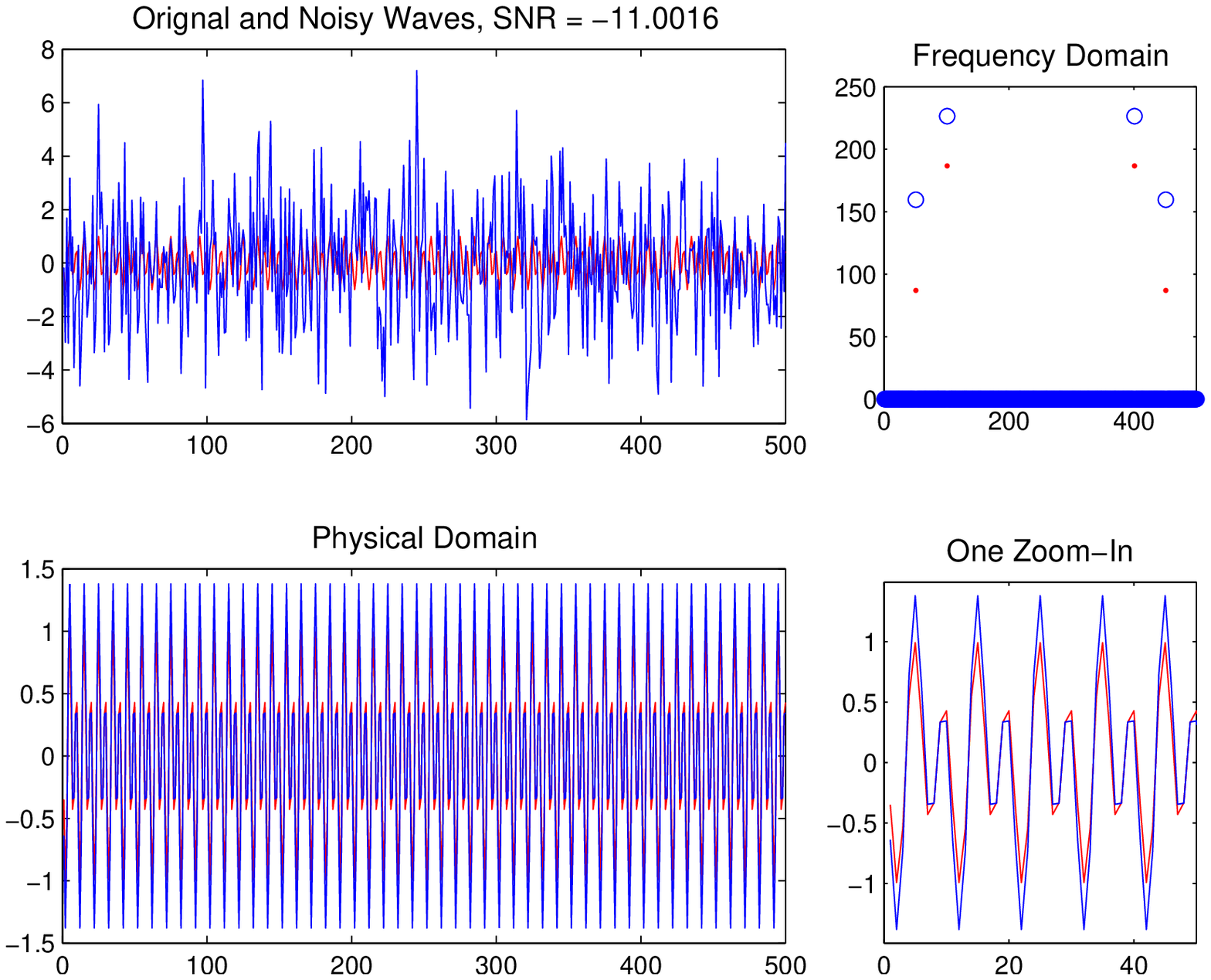}
    \caption{Reconstruction using 80\% random samples of $\tilde u$
with SNR$=-11.0016$. The upper left figure
    shows the original (red) and noisy (blue) signals; the upper right shows the reconstruction (blue circle) v.s. original signal
    (red dots) in Fourier domain in terms of their magnitudes (i.e. $|\widehat u^*|$ v.s. $|\widehat {\bar u}|$);
    bottom left shows the reconstructed (blue) v.s. original (red) signal in physical domain; and bottom right shows
    one close-up of the figure at bottom left.}\label{Fig:sinu4}
\end{figure}
\end{document}